\documentclass[11pt]{article}
\usepackage{mathrsfs}
\usepackage{amsfonts}
\usepackage{amsmath}
\usepackage{amssymb}
\usepackage{amsopn}
\usepackage{amsthm}
\usepackage{amstext}

\def\ep{\epsilon}
\def\de{\delta}

\def\pn{(P_1,\cdots,P_n)}

\def\m{\mathrm M}
\def\inn{i=1,\cdots,n}
\def\aa{\mathcal{A}}
\def\Ran{\mathrm{Ran}}
\def\Ker{\mathrm{Ker}}
\def\P{\mathbf{P}\!}
\def\PC{\mathbf{PC}}
\def\PI{\mathbf{PI}}
\def\PO{\mathbf{PO}}
\def\Tr{\mathrm{Tr}}

\newtheorem{lemma}{Lemma}[section]
\newtheorem{theorem}[lemma]{Theorem}
\newtheorem{definition}[lemma]{Definition}
\newtheorem{proposition}[lemma]{Proposition}
\newtheorem{corollary}[lemma]{Corollary}
\newtheorem{remark}[lemma]{Remark}
\newtheorem{example}[lemma]{Example}

\begin{document}
\date{}
\title{Completeness of $n$--tuple of projections in $C^*$--algebras}
\author{Shanwen Hu\thanks{E-mail: swhu@math.ecnu.edu.cn}\hspace*{2mm} and\
 Yifeng Xue\thanks{E-mail: yfxue@math.ecnu.edu.cn} \\
 Department of mathematics and Research Center for Operator Algebras\\
 East China Normal University, Shanghai 200241, P.R. China
}

\maketitle
\setlength{\baselineskip}{15.6pt}
\begin{abstract} \noindent
Let $(P_1,\cdots,P_n)$ be an $n$--tuple of projections in a unital $C^*$--algebra $\aa$. We say $\pn$ is complete in
$\aa$ if $\aa$ is the linear direct sum of the closed subspaces $P_1\aa,\cdots,P_n\aa$. In this paper, we give some necessary and sufficient conditions
for the completeness of $\pn$ and discuss the perturbation problem and topology of the set of all complete $n$--tuple of
projections in $\aa$. Some interesting and significant results are obtained in this paper.

\vspace{3mm}

\end{abstract}

\setcounter{section}{-1}
\section{Introduction}

Throughout the paper, we always assume that $H$ is a complex Hilbert space with inner product $<\cdot,\cdot>$, $B(H)$ is
the $C^*$--algebra of all bounded linear operators on $H$ and $\aa$ is a $C^*$--algebra with the unit $1$. Let $\aa_+$
denote the set of all positive elements in $\aa$. It is well--known that $\aa$ has a faithful representation
$(\psi ,H_\psi )$ with $\psi (1)=I$ (cf. \cite[Theorem 1.6.17]{HL} or \cite[Theorem 1.5.36]{Xue}). For $T\in B(H)$, let
$\Ran(T)$ (resp. $\Ker(T)$) denote the range (resp. kernel) of $T$.

Let $V_1,V_2$ be closed subspaces in $H$ such that $H=V_1\dotplus V_2=V_1^\perp\dotplus V_2$, that is, $V_1$ and $V_2$ is
in generic position (cf. \cite{Ha}). Let $P_i$ be projection of $H$ onto $V_i$, $i=1,2$. Then
$H=\Ran(P_1)\dotplus\Ran(P_2)=\Ran(I-P_1)\dotplus\Ran(P_2)$. In this case, Halmos gave very useful matrix representations
of $P_1$ and $P_2$ in \cite{Ha}. Following Halmos' work, Sunder investigated in \cite{Su} the $n$--tuple closed subspaces
$(V_1,\cdots,V_n)$ in $H$ which satisfying the condition $H=V_1\dotplus\cdots\dotplus V_n$. If let $P_i$ be the projection
of $H$ onto $V_i$, $i=1,\cdots,n$, then the condition $H=V_1\dotplus\cdots\dotplus V_n$ is equivalent to
$H=\Ran(P_1)\dotplus\cdots\dotplus\Ran(P_n)$.

Now the question yields: when does the relation $H=\Ran(P_1)\dotplus\cdots\dotplus\Ran(P_n)$ hold for an $n$--tuple of
projections $(P_1,\cdots,P_n)$? When $n=2$, Buckholdtz proved in \cite{BU} that $\Ran(P_1)\dotplus \Ran(P_2)=H$ iff
$P_1-P_2$ is invertible in $B(H)$ iff $I-P_1P_2$ is invertible in $B(H)$ and iff $P_1+P_2-P_1P_2$ is invertible in $B(H)$.
More information about two projections can be found in \cite{BS}.
Koliha and Rako\v{c}evi\'{c} generalized Buckholdtz's work to the set of $C^*$--algebras and rings. They gave some
equivalent conditions for decomposition $\mathfrak R=P\,\mathfrak R\dotplus Q\,\mathfrak R$ or
$\mathfrak R=\mathfrak R\,P\dotplus \mathfrak R\,Q$ in \cite {KO} and \cite{KO2} for idempotent elements $P$ and $Q$ in
a unital ring $\mathfrak R$. They also characterized the Fredhomness of the difference of projections on $H$ in \cite{KO3}.
For $n\ge 3$, the question remains unknown so far. But there are some works concerning with this problem. For example,
the estimation of the spectrum of the finite sum of projections on $H$ is given in \cite{BM} and the $C^*$--algebra
generated by certain projections is investigated in \cite{Sh} and \cite{V}, etc..

Let $\P_n(\aa)$ denote the set of $n$--tuple ($n\ge 2$) of non--trivial projections in $\aa$ and put
$$
\PC_n(\aa)=\{(P_1,\cdots,P_n)\in\P_n(\aa)\,\vert\,P_1\aa\dotplus\cdots\dotplus P_n\aa=\aa\}.
$$
It is worth to note that if $\aa=B(H)$ and $(P_1,\cdots,P_n)\in\P_n(B(H))$, then $(P_1,\cdots,P_n)$ $\in\PC_n(B(H))$ if and only if
$\Ran(P_1)\dotplus\cdots\dotplus\Ran(P_n)=H$ (see Theorem \ref{th1} below).

In this paper, we will investigate the set $\PC_n(\aa)$ for $n\ge 3$.
The paper consists of four sections. In Section 1, we give some necessary and sufficient
conditions that make $(P_1,\cdots,P_n)\in\P_n(\aa)$ be in $\PC_n(\aa)$. In Section 2, using some equivalent conditions
for $(P_1,\cdots,P_n)\in\PC_n(\aa)$ obtained in \S 1, we obtain an explicit expression of $P_{i_1}\vee\cdots\vee P_{i_k}$
for $\{i_1,\cdots,i_k\}\subset\{1,\cdots,n\}$. We discuss the perturbation problems for $(P_1,\cdots,P_n)\in\PC_n(\aa)$ in
Section 3. We find an interesting result: if $\pn\in\P_n(\aa)$ with $A=\sum\limits_{i=1}^nP_i$ invertible in $\aa$, then
$\|P_iA^{-1}P_j\|<\big[(n-1)\|A^{-1}\|\|A\|^2\big]^{-1}$, $i\not=j$ implies $P_iA^{-1}P_j=0$, $i\not=j$, $i,j=1,\cdots,n$
in this section. We show in this section that for given $\ep\in (0,1)$, if $\pn\in\P_n(\aa)$ satisfies condition
$\|P_iP_j\|<\ep$, then there exists an $n$--tuple of mutually orthogonal projections $(P'_1,\cdots,P'_n)\in\P_n(\aa)$ such
that $\|P_i-P'_i\|<2(n-1)\ep$, $\inn,$ which improves a conventional  estimate: $\|P_i-P'_i\|<(12)^{n-1}n!\ep$,
$i=1,\cdots,n$ (cf. \cite{HL}). In the final section, we will study the topological properties and equivalent relations
on $\PC_n(\aa)$.

\section{Equivalent conditions for complete $n$--tuples of projections in $C^*$--algebras}

Let $GL(\aa)$ (resp. $U(\aa)$) denote the group of all invertible (resp. unitary) elements in $\aa$. Let $\m_k(\aa)$
denote matrix algebra of all $k\times k$ matrices over $\aa$. For any $a\in\aa$, we set $a\,\aa=\{ax\vert\,x\in\aa\}
\subset\aa$.

\begin{definition}\label{1da}
An $n$--tuple of projections $(P_1,\cdots,P_n)$ in $\aa$ is called complete in $\aa$, if
$(P_1,\cdots,P_n)\in\PC_n(\aa)$.
\end{definition}

\begin{theorem}\label{th1}
Let $(P_1,\cdots,P_n)\in\P_n(\aa)$. Then the following statements are equivalent:
\begin{enumerate}
\item[$(1)$] $(P_1,\cdots,P_n)$ is complete in $\aa$.
\item[$(2)$] $H_\psi =\Ran(\psi (P_1))\dotplus\cdots\dotplus\Ran(\psi (P_n))$ for any faithful representation $(\psi , H_\psi )$
of $\aa$ with $\psi (1)=I$.
\item[$(3)$] $H_\psi =\Ran(\psi (P_1))\dotplus\cdots\dotplus\Ran(\psi (P_n))$ for some faithful representation $(\psi , H_\psi )$
of $\aa$ with $\psi (1)=I$.
\item[$(4)$] $\sum\limits_{j\not=i}P_j+\lambda P_i\in GL(\aa)$, $i=1,2,\cdots,n$ and
$\forall\,\lambda\in [1-n,0)$.
\item[$(5)$] $\lambda\big(\sum\limits_{j\not=i}P_j\big)+ P_i\in GL(\aa)$ for $1\le i\le n$ and all
$\lambda\in\mathbb C\backslash\{0\}$.
\item[$(6)$] $A=\sum\limits_{i=1}^nP_i\in GL(\aa)$ and $P_iA^{-1}P_i=P_i$, $i=1,\cdots,n$.
\item[$(7)$] $A=\sum\limits_{i=1}^nP_i\in GL(\aa)$ and $P_iA^{-1}P_j=0$, $i\not=j$, $i,j=1,\cdots,n$.
\item[$(8)$] there is an $n$-tuple of idempotent operators ${(E_1,\cdots,E_n)}$ in $\aa$ such that
$E_iP_i=E_i,\,P_iE_i=P_i$, $i=1,\cdots,n$ and $E_iE_j=0,\ i\not=j,\ i,j=1,\cdots,n,\ \sum\limits_{i=1}^nE_i=1.$
\end{enumerate}
\end{theorem}

 In order to show Theorem \ref{th1}, we need following lemmas.

\begin{lemma}\label{1L1}
Let $B,\,C\in\aa_+\backslash\{0\}$ and suppose that $\lambda B+C$ is invertible in $\aa$ for every
$\lambda\in \mathbb R\backslash\{0\}$. Then there is a non--trivial orthogonal projection $P\in\aa$ such that
$$
B=(B+C)^{1/2}P(B+C)^{1/2},\quad C=(B+C)^{1/2}(1-P)(B+C)^{1/2}.
$$
\end{lemma}
\begin{proof}Put $D=B+C$ and $D_\lambda=\lambda B+C$, $\forall\,\lambda\in R\backslash\{0\}$. Then $D\ge 0$, $D$ and
$D_\lambda$ are all invertible in $\aa$, $\forall\,\lambda\in\mathbb R\backslash\{0\}$.

Put $B_1=D^{-1/2}BD^{-1/2}$, $C_1=D^{-1/2}CD^{-1/2}$. Then $B_1+C_1=1$ and
$$
D^{-1/2}D_\lambda D^{-1/2}=\lambda B_1+C_1=\lambda +(1-\lambda)C_1=(1-\lambda)(\lambda(1-\lambda)^{-1}+C_1)
$$
is invertible in $\aa$ for any $\lambda\in\mathbb R\backslash\{0,1\}$. Since $\lambda\mapsto\dfrac{\lambda}{1-\lambda}$ is
a homeomorphism from $\mathbb R\backslash\{0,1\}$ onto $\mathbb R\backslash\{-1,0\}$, it follows that $\sigma(C_1)
\subset\{0,1\}$. Note that $B_1$ and $C_1$ are all non--zero. So $\sigma(C_1)=\{0,1\}=\sigma(B_1)$ and hence
$P=B_1$ is a non--zero projection in $\aa$ and $B=D^{1/2}PD^{1/2}$, $C=D^{1/2}(1-P)D^{1/2}$.
\end{proof}

\begin{lemma}\label{1L2}
Let $B,\,C\in\aa_+\backslash\{0\}$.  Then the following statements are equivalent:
\begin{enumerate}
\item[$(1)$] for any non--zero real number $\lambda$, $\lambda B+C$ is invertible in $\aa$.
\item[$(2)$] $B+C$ is invertible in $\aa$ and $B(B+C)^{-1}B=B$.
\item[$(3)$] $B+C$ is invertible in $\aa$ and $B(B+C)^{-1}C=0$.
\item[$(4)$] $B+C$ is invertible in $\aa$ and for any $B',\,C'\in\aa_+$ with $B'\le B$ and $C'\le C$,
$B'(B+C)^{-1}C'=0$.
\end{enumerate}
\end{lemma}
\begin{proof} (1)$\Rightarrow$(2) By Lemma \ref{1L1}, there is a non--zero projection $P$ in $\aa$ such that
$B=D^{1/2}PD^{1/2}$, $C=D^{1/2}(1-P)D^{1/2}$, where $D=B+C\in GL(\aa)$. So
$$
B(B+C)^{-1}B=D^{1/2}PD^{1/2}D^{-1}D^{1/2}PD^{1/2}=B.
$$

The assertion (2) $\Leftrightarrow$ (3) follows from
$$
B(B+C)^{-1}B=B(B+C)^{-1}(B+C-C)=B-B(B+C)^{-1}C.
$$

(3) $\Rightarrow$ (4) For any $C'$ with $0\le C'\le C$,
$$
0\le B(B+C)^{-1}C'(B+C)^{-1}B\le B(B+C)^{-1}C(B+C)^{-1}B=0,
$$
we have $B(B+C)^{-1}C'=B(B+C)^{-1}C'^{1/2}C'^{1/2}=0.$ This implies $C'(B+C)^{-1}B=0$. In the same way, we get that for
any $B'$ with $0\le B'\le B$, $C'(B+C)^{-1}B'=0$.

(4)$\Rightarrow$(3) is obvious.

(2)$\Rightarrow$(1) Set $X=(B+C)^{-1/2}B$ and $Y=(B+C)^{-1/2}C$. Then $X,\,Y\in\aa$ and $X^*X=B$, $X+Y=(B+C)^{1/2}$.
Thus, for any $\lambda\in\mathbb R\backslash\{0\}$,
\begin{align*}
X+\lambda Y&=(B+C)^{-1/2}(B+\lambda C)\\
(X+\lambda Y)^*(X+\lambda Y)&=((1-\lambda)X+\lambda(B+C)^{1/2})^*((1-\lambda)X+\lambda(B+C)^{1/2})\\
&=(1-\lambda)^2 B+2\lambda(1-\lambda)B+\lambda^2(B+C)\\
&=B+\lambda^2 C
\end{align*}
and consequently, $(X+\lambda Y)^*(X+\lambda Y)\ge B+C$ if $|\lambda|>1$ and
$(X+\lambda Y)^*(X+\lambda Y)\ge\lambda^2(B+C)$ when $|\lambda|<1$. This indicates that $(X+\lambda Y)^*(X+\lambda Y)$
is invertible in $\aa$. Noting that $B+C\ge\|(B+C)^{-1}\|^{-1}\cdot 1$, we have, for any $\lambda\in\mathbb R\backslash\{0\}$,
\begin{align*}
(B+\lambda C)^2&=(X+\lambda Y)^*(B+C)(X+\lambda Y)\\
&\ge\|(B+C)^{-1}\|^{-1}(X+\lambda Y)^*(X+\lambda Y).
\end{align*}
Therefore, $B+\lambda C$ is invertible in $\aa$, $\forall\,\lambda\in\mathbb R\backslash\{0\}$.
\end{proof}

Now we begin to prove Theorem \ref{th1}.

(1)$\Rightarrow$(6) Statement (1) implies that there are $b_1,\cdots,b_n\in\aa$ such that $1=\sum\limits^n_{i=1}P_ib_i$.
Put $\hat I=\begin{bmatrix}1\\ \ &0\\ \ &\ &\ddots\\ \ &\ &\ \ &0\end{bmatrix}$,
$X=\begin{bmatrix}P_1&\cdots& P_n\\ 0&\cdots &0\\ \vdots &\ddots&\vdots\\ 0&\cdots &0\end{bmatrix}$ and
$Y=\begin{bmatrix}b_1&0&\cdots&0\\ \vdots&\vdots&\ddots&\vdots\\ b_n&0&\cdots&0\end{bmatrix}$. Then
$$
\hat I=XY=XYY^*X^*\le\|Y\|^2XX^*=\|Y\|^2\begin{bmatrix}\sum\limits^n_{i=1}P_i\\ \ &0\\
\ &\ &\ddots\\ \ &\ &\ \ &0\end{bmatrix}
$$
and so that $A=\sum\limits^n_{i=1}P_i$ is invertible in $\aa$. Therefore, from $\aa=P_1\aa\dotplus\cdots\dotplus P_n\aa$
and
$$
P_i=P_1A^{-1}P_i+\cdots+P_iA^{-1}P_i+\cdots+P_nA^{-1}P_i=\underbrace{0+\cdots+0}_{i-1}+P_i+\underbrace{0+\cdots+0}_{n-i},
$$
$i=1,\cdots,n$, we get that $P_i=P_iA^{-1}P_i$, $i=1,\cdots,n$.

(2)$\Rightarrow$(3) is obvious.

(3)$\Rightarrow$(4) Set $Q_i=\psi (P_i)$, $i=1,\cdots,n$. From $H_\psi =\Ran(Q_1)\dotplus\cdots\dotplus \Ran(Q_n)$, we obtain
idempotent operators $F_1,\cdots,F_n$ in $B(H_\psi )$ such that $\sum\limits^n_{i=1}F_i=I$, $F_iF_j=0$, $i\not=j$ and
$F_iH_\psi =Q_iH_\psi $, $i,j=1,\cdots,n$. So $F_iQ_i=Q_i$, $Q_iF_i=F_i$ and
$F_jQ_i=0$, $i\not=j$, $1\le i,j\le n$. Using these relations, it is easy to check that
\begin{align*}
\big(\sum_{i=1}^n\lambda_iQ_i\big)\big(\sum_{i=1}^n\lambda_i^{-1}F_i^*F_i\big)&=\sum_{i=1}^nF_i=I,\\
\big(\sum_{i=1}^n\lambda_i^{-1}F_i^*F_i\big)\big(\sum_{i=1}^n\lambda_iQ_i\big)&=\sum_{i=1}^nF_i^*=I,
\end{align*}
for any non--zero complex number $\lambda_i$, $i=1,\cdots,n$. Particularly, for any $\lambda\in [1-n,0)$,
$$
\big(\lambda\big(\sum\limits_{j\not=i}Q_j\big)+ Q_i\big)^{-1}=\lambda^{-1}\sum_{j\not=i}F_j^*F_j+F_i^*F_i
$$
in $B(H_\psi )$. Thus, $\lambda\big(\sum\limits_{j\not=i}Q_j\big)+Q_i$ is invertible $\psi (\aa)$, $1\le i\le n$ by
\cite[Corollary 1.5.8]{Xue} and so that $\lambda\big(\sum\limits_{j\not=i}P_j\big)+P_i\in GL(\aa)$ since $\psi $ is faithful
and $\psi (1)=I$.

(4)$\Rightarrow$(5) Put $A_i(\lambda)=\sum\limits_{j\not=i}P_j+\lambda P_i$, $i=1,\cdots,n$,
$\lambda\in\mathbb R\backslash\{0\}$, then
\begin{align*}
(A_i(\lambda))^2&\le 2\big(\sum\limits_{j\not=i}P_j\big)^2+2\lambda^2 P_i\le 2(n-1)\sum\limits_{j\not=i}P_j
+2\lambda^2 P_i\\
&\le 2\max\{n-1,\lambda^2\}(P_1+\cdots+P_n).
\end{align*}
So $A_i(\lambda)$ is invertible in $\aa$, $\forall\,\lambda\in [1-n,0)$ means that
$A=P_1+\cdots+P_n$ is invertible in $\aa$. Consequently, $A_i(\lambda)$ is invertible in $\aa$ when $\lambda>0$,
$\forall\,1\le i\le n$.

Now we show that $A_i(\lambda)$ is invertible in $\aa$ for $i=1,\cdots,n$ and $\lambda <1-n$. Put
$$
A_{1i}=P_iAP_i,\ A_{2i}=P_iA(1-P_i),\ A_{4i}=(1-P_i)A(1-P_i),\ i=1,\cdots,n.
$$
Express $A_i(\lambda)$ as the form $A_i(\lambda)=\begin{bmatrix}A_{1i}+(\lambda-1)P_i&A_{2i}\\ A_{2i}^*&A_{4i}
\end{bmatrix}$, $i=1,\cdots,n$. Noting that $A_{4i}$ is invertible in $(1-P_i)\aa(1-P_i)$ ($A\ge\|A^{-1}\|^{-1}\cdot 1$,
$A_{4i}\ge\|A^{-1}\|^{-1}(1-P_i)$) and
$$
A_i(\lambda)\begin{bmatrix}P_i&0\\ -A_{4i}^{-1}A_{2i}^*&1-P_i\end{bmatrix}
=\begin{bmatrix}A_{1i}-A_{2i}A_{4i}^{-1}A_{2i}^*+(\lambda-1)P_i&A_{2i}\\ 0&A_{4i}\end{bmatrix},
$$
we get that $A_i(\lambda)$ is invertible iff $A_{1i}-A_{2i}A_{4i}^{-1}A_{2i}^*+(\lambda-1)P_i$
is invertible in $P_i\aa P_i$, $i=1,\cdots,n$. Since $A_{1i}\le nP_i$, it follows that
$$
-A_{1i}+A_{2i}A_{4i}^{-1}A_{2i}^*-(\lambda-1)P_i\ge (1-n-\lambda)P_i+A_{2i}A_{4i}^{-1}A_{2i}^*\ge(1-n-\lambda)P_i
$$
when $\lambda<1-n$, $i=1,\cdots,n$. Therefore, $A_i(\lambda)$ is invertible in $\aa$ for $\lambda<1-n$ and $i=1,\cdots,n$.

Applying Lemma \ref{1L2} to $\sum\limits_{j\not=i}P_j$ and $P_i$, $1\le i\le n$, we can get the implications
(5)$\Rightarrow$(6) and (6)$\Rightarrow$ (7) easily.

(7)$\Rightarrow$(8) Set $E_i=P_iA^{-1}$, $i=1\cdots,n$. Then $E_i$ is an idempotent elements in $\aa$
and $E_iE_j=0$, $i\not=j$, $i,j=1,\cdots,n$. It is obvious that $\sum\limits_{i=1}^n E_i=1$ and $P_iE_i=E_i$, $E_iP_i=P_i$,
$i=1,\cdots,n$.

(8)$\Rightarrow$(1) Let $E_1,\cdots,E_n$ be idempotent elements in $\aa$ such that $E_iE_j=\delta_{ij}E_i$,
$\sum\limits^n_{i=1}E_i=1$ and $E_iP_i=P_i$, $P_iE_i=E_i$, $i,j=1,\cdots,n$. Then $E_i\aa=P_i\aa$, $i=1,\cdots,n$
and $\aa=E_1\aa\dotplus\cdots E_n\aa=P_1\aa\dotplus\cdots\dotplus P_n\aa$.

(8)$\Rightarrow$(2) Let $E_1,\cdots,E_n$ be idempotent elements in $\aa$ such that $E_iE_j=\delta_{ij}E_i$,
$\sum\limits^n_{i=1}E_i=1$ and $E_iP_i=P_i$, $P_iE_i=E_i$, $i,j=1,\cdots,n$. Let $(\psi ,H_\psi )$ be any faithful
representation of $\aa$ with $\psi (1)=I$. Put $E_i'=\psi (E_i)$ and $Q_i=\psi (P_i)$, $i=1,\cdots,n$. Then
$E_i'E_j'=\delta_{ij}E_i'$, $\sum\limits^n_{i=1}E_i'=I$ and $\Ran(E_i')=\Ran(Q_i)$, $i,j=1,\cdots,n$. Consequently,
$H_\psi =\Ran(Q_1)\dotplus\cdots\dotplus\Ran(Q_n)$. \qed

\begin{remark}\label{rem1a}
{\rm
(1) Statement  (3) in Theorem  \ref{th1} can not be replaced by
``for any $1\le i\le n$, $P_i-\sum\limits_{j\not=i}P_j$ is invertible".

For example, let $H^{(4)}=\bigoplus\limits^4_{i=1}H$ and put $\aa=B(H^{(4)})$,
$$
P_1=\begin{bmatrix}I\\ \ & I\\ \ &\ &0\\ \ &\ &\ & 0\end{bmatrix},\quad
P_2=\begin{bmatrix}I\\ \ & 0\\ \ &\ &I\\ \ &\ &\ &0\end{bmatrix},\quad
P_3=\begin{bmatrix}I\\ \ &0\\ \ &\ &0\\ \ &\ &\ & I\end{bmatrix}.
$$
Clearly, $P_i-\sum\limits_{j\not=i}{P_j}$ is invertible, $1\le i\le 3$, but $P_2+P_3-2P_1$ is not
invertible, that is, $(P_1,P_2,P_3)$ is not complete in $\aa$.

(2) According to the proof of (3)$\Rightarrow$(4) of Theorem \ref{th1}, we see that for $(P_1,\cdots,P_n)$ $\in\P_n(\aa)$,
if $\sum\limits^n_{i=1}P_i\in GL(\aa)$, then $\sum\limits_{i\not=j}P_i-\lambda P_i\in GL(\aa)$, $\forall\,1\le i\le n$
and $\lambda>n-1$.
}
\end{remark}
\begin{corollary}[{\cite[Theorem 1]{BU}}]
Let $P_1,P_2$ be non--trivial projections in $B(H)$. Then $H=\Ran(P_1)\dotplus\Ran(P_2)$ iff  $P_1-P_2$ is invertible in
$B(H)$.
\end{corollary}
\begin{proof}By Theorem \ref{th1}, $H=\Ran(P_1)\dotplus\Ran(P_2)$ implies that $P_1-P_2\in GL(B(H))$.

Conversely, if $P_1-P_2\in GL(B(H))$, then from $2(P_1+P_2)\ge(P_1-P_2)^2$, we get that $P_1+P_2\in GL(B(H))$ and
so that $P_1-\lambda P_2, P_2-\lambda P_1\in GL(B(H))$, $\forall\,\lambda>1$ by Remark \ref{rem1a} (2). Thus, for any
$\lambda\in (0,1]$, $P_1-\lambda P_2$ and $P_2-\lambda P_1$ are all invertible in $B(H)$. Consequently,
$H=\Ran(P_1)\dotplus\Ran(P_2)$ by Theorem \ref{th1}.
\end{proof}

\section{Some representations concerning the complete $n$--tuple of projections}

We first statement two lemmas which are frequently used in this section and the later sections.

\begin{lemma}\label{lem3a}
Let $B\in\aa_+$ such that $0\in\sigma(B)$ is an isolated point. Then there is a unique element $B^\dag\in\aa_+$ such that
$$
BB^\dag B=B,\ B^\dag BB^\dag=B^\dag,\ BB^\dag=B^\dag B.
$$
\end{lemma}
\begin{proof} The assertion follows from Proposition 3.5.8, Proposition 3.5.3 and Lemma 3.5.1 of \cite{Xue}.
\end{proof}

\begin{remark}
{\rm
The element $B^\dag$ in Lemma \ref{lem3a} is called the Moore--Penrose inverse of $B$. When $0\not\in\sigma(B)$,
$B^\dag\triangleq B^{-1}$. The detailed information can be found in \cite{Xue}.
}
\end{remark}

The following lemma comes from \cite[Lemma 3.5.5]{Xue} and \cite[Lemma 1]{CX}:
\begin{lemma}\label{lem3b}
Let $P\in\aa$ be an idempotent element. Then
\begin{enumerate}
\item[$(1)$] $P+P^*-1\in GL(\aa)$.
\item[$(2)$] $R=P(P+P^*-1)^{-1}$ is a projection in $\aa$ satisfying $PR=R$ and $RP=P$.
\end{enumerate}
Moreover, if $R'\in\aa$ is a projection such that $PR'=R'$ and $R'P=P$, then $R'=R$.
\end{lemma}

Let $(P_1,\cdots,P_n)\in\PC_n(\aa)$ and put $A=\sum\limits^n_{i=1}P_i$. By Theorem \ref{th1}, $A\in GL(\aa)$ and
$E_i=P_iA^{-1}$, $1\le i\le n$ are idempotent elements satisfying conditions
$$
E_iE_j=0,\ i\not=j,\ E_iP_i=P_i,\ P_iE_i=E_i,\ i=1,\cdots,n,\ \text{and}\ \sum\limits^n_{i=1}E_i=1.
$$
By Lemma \ref{lem3b}, $P_i=E_i(E_i^*+E_i-1)^{-1}$, $1\le i\le n$.
So the $C^*$--algebra $C^*(P_1,\cdots,P_n)$ generated by $P_1,\cdots,P_n$ is equal to the $C^*$--algebra
$C^*(E_1,\cdots,E_n)$ generated by $E_1,\cdots,E_n$.

Put $Q_i=A^{-1/2}P_iA^{-1/2}$, $i=1,\cdots,n$. Then $Q_iQ_j=\delta_{ij}Q_i$ by Theorem \ref{th1}, $i,j=1,\cdots,n$
and $\sum\limits_{i=1}^nQ_i=1$. Thus,
\begin{equation}\label{3eqa}
P_i=A^{1/2}Q_iA^{1/2}\ \text{and}\ E_i=P_iA^{-1}=A^{1/2}Q_iA^{-1/2},\ i=1,\cdots,n.
\end{equation}

\begin{proposition}\label{prop3a}
Let $(P_1,\cdots,P_n)\in\PC_n(\aa)$ with $A=\sum\limits^n_{i=1}P_i$. Then for any $\lambda_i\not=0$, $i=1,\cdots,n$,
$\big(\sum\limits_{i=1}^n\lambda_iP_i\big)^{-1}=A^{-1}\big(\sum\limits_{i=1}^n\lambda_i^{-1}P_i\big)A^{-1}.$
\end{proposition}
{\it\noindent Proof.}\ \,Keeping the symbols as above. We have
$\sum\limits_{i=1}^n\lambda_iP_i=A^{1/2}\big(\sum\limits^n_{i=1}\lambda_i Q_i\big)A^{1/2}$. Thus,
$$
\hspace{1.7cm}\big(\sum\limits_{i=1}^n\lambda_iP_i\big)^{-1}=A^{-1/2}\big(\sum\limits^n_{i=1}\lambda_i^{-1} Q_i\big)A^{-1/2}
=A^{-1}\big(\sum_{i=1}^n\lambda_i^{-1}P_i\big)A^{-1}.\hspace{1cm}\qed
$$

Now for $i_1,i_2,\cdots,i_k\in\{1,2,\cdots,n\}$ with $i_1<i_2<\cdots<i_k$, put $A_0=\sum\limits^k_{r=1}P_{i_r}$ and
$Q_0=\sum\limits^k_{r=1}Q_{i_r}$. Then $A_0,Q_0\in\aa$ and $Q_0$ is a projection. From (\ref{3eqa}),
$A_0=A^{1/2}Q_0A^{1/2}$. Thus, $\sigma(A_0)\backslash\{0\}
=\sigma(Q_0AQ_0)\backslash\{0\}$ (cf. \cite[Proposition 1.4.14]{Xue}). Since $Q_0AQ_0$ is invertible in $Q_0\aa Q_0$,
it follows that $0\in\sigma(Q_0AQ_0)$ is an isolated point and so that $0\in\sigma(A_0)$ is also an isolated point.
So we can define $P_{i_1}\vee\cdots\vee P_{i_k}$ to be the projection $A_0^\dag A_0\in\aa$ by Lemma \ref{lem3a}.
This definition is reasonable:

if $P\in\aa$ is a projection such that $P\ge P_{i_r}$, $r=1,\cdots,k$, then $PA_0=A_0$
and hence $PA_0A_0^\dag=A_0A_0^\dag$, i.e., $P\ge P_{i_1}\vee\cdots\vee P_{i_k}$; Since $A_0\ge P_{i_r}$, we have
$$
0=(1-A_0^\dag A_0)A_0(1-A_0^\dag A_0)\ge(1-A_0^\dag A_0)P_{i_r}(1-A_0^\dag A_0)
$$
and consequently, $P_{i_r}(1-A_0^\dag A_0)=0$, that is, $P_{i_r}\le P_{i_1}\vee\cdots\vee P_{i_k}$, $i=1,\cdots,k$.

\begin{proposition} \label{prop3b}
Let $(P_1,\cdots,P_n)\in\PC_n(\aa)$ with $A=\sum\limits^n_{i=1}P_i$. Let $i_1,\cdots,i_k$ be as above and
$\{j_1,\cdots,j_l\}=\{1,\cdots,n\}\backslash\{i_1,\cdots,i_k\}$ with $j_1<\cdots<j_l$.
Then
\begin{align}
\label{3eqb} P_{i_1}\vee\cdots \vee P_{i_k}&=
A^{1/2}\big[\big(\sum\limits_{r=1}^k Q_{i_r}\big)A\big(\sum\limits_{r=1}^k Q_{i_r}\big)\big]^{-1}A^{1/2}\\
\label{3eqc} &=\big(\sum_{r=1}^kP_{i_r}\big)\big[\big(\sum_{r=1}^kP_{i_r}\big)^2+\sum_{t=1}^lP_{j_t}\big]^{-1}
\big(\sum_{r=1}^kP_{i_r}\big).
\end{align}
\end{proposition}
\begin{proof} Using the symbols $P_i, Q_i, E_i$ as above. According to (\ref{3eqa}),
$$
\sum\limits_{r=1}^k P_{i_r}=A^{1/2}\big(\sum\limits_{r=1}^k Q_{i_r}\big)A^{1/2},\
\sum\limits_{r=1}^k E_{i_r}=A^{1/2}\big(\sum\limits_{r=1}^k Q_{i_r}\big)A^{-1/2}.
$$
Thus $\big(\sum\limits_{r=1}^k E_{i_r}\big)\big(\sum\limits_{r=1}^k P_{i_r}\big)=\sum\limits_{r=1}^k P_{i_r}$
and $\sum\limits_{r=1}^k E_{i_r}=\big(\sum\limits_{r=1}^k P_{i_r}\big)A^{-1}$. Then we have
$$
\big(\sum\limits_{r=1}^k E_{i_r}\big)P_{i_1}\vee\cdots \vee P_{i_k}=P_{i_1}\vee\cdots \vee P_{i_k},\quad
P_{i_1}\vee\cdots \vee P_{i_k}\big(\sum\limits_{r=1}^k E_{i_r}\big)=\sum\limits_{r=1}^k E_{i_r},
$$
according to the definition of $P_{i_1}\vee\cdots \vee P_{i_k}$.

Since
$\sum\limits_{r=1}^k E_{i_r}$ is an idempotent element in $\aa$, it follows from Lemma \ref{lem3b} that
\begin{equation}\label{3eqd}
P_{i_1}\vee\cdots \vee P_{i_k}=\big(\sum\limits_{r=1}^k E_{i_r}\big)
\big[\sum\limits_{r=1}^k (E_{i_r}^*+E_{i_r})-1\big]^{-1}\in\aa.
\end{equation}
Noting that
$\big(\sum\limits_{r=1}^k Q_{i_r}\big)A\big(\sum\limits_{r=1}^k Q_{i_r}\big)\in
GL\big(\big(\sum\limits_{r=1}^k Q_{i_r}\big)\aa\big(\sum\limits_{r=1}^k Q_{i_r}\big)\big)$;
$\big(\sum\limits_{t=1}^l Q_{j_t}\big)A\big(\sum\limits_{t=1}^k Q_{j_t}\big)$ is invertible in
$\big(\sum\limits_{t=1}^k Q_{j_t}\big)\aa\big(\sum\limits_{t=1}^k Q_{j_t}\big)$ and
\begin{align*}
\sum\limits_{r=1}^k (E_{i_r}^*+E_{i_r})-1&=A^{-1/2}\big[\big(\sum\limits_{r=1}^k Q_{i_r}\big)A+
A\big(\sum\limits_{r=1}^k Q_{i_r}\big)-A\big]A^{-1/2}\\
&=A^{-1/2}\big[\big(\sum\limits_{r=1}^k Q_{i_r}\big)A\big(\sum\limits_{r=1}^k Q_{i_r}\big)\!-\!
\big(\sum\limits_{t=1}^l Q_{j_t}\big)A\big(\sum\limits_{t=1}^l Q_{j_t}\big)\big]A^{-1/2},
\end{align*}
we obtain that
\begin{align*}
\big[\sum\limits_{r=1}^k (E_{i_r}^*&+E_{i_r})-1\big]^{-1}\\
&=A^{1/2}\big[\big[\big(\sum\limits_{r=1}^k Q_{i_r}\big)A\big(\sum\limits_{r=1}^k Q_{i_r}\big)\big]^{-1}\!-\!
\big[\big(\sum\limits_{t=1}^l Q_{j_t}\big)A\big(\sum\limits_{t=1}^l Q_{j_t}\big)\big]^{-1}\big]A^{1/2}.
\end{align*}
Combining this with (\ref{3eqd}), we can get (\ref{3eqb}).

Note that $\sum\limits_{r=1}^kP_{i_r}=A^{1/2}\big(\sum\limits_{r=1}^kQ_{i_r}\big)A^{1/2}$,
$\sum\limits_{t=1}^lP_{j_t}=A^{1/2}\big(\sum\limits_{t=1}^lQ_{j_t}\big)A^{1/2}$ and
$
\big(\sum\limits_{r=1}^kP_{i_r}\big)^2$ $=A^{1/2}\big(\sum\limits_{r=1}^kQ_{i_r}\big)A\big(\sum\limits_{r=1}^kQ_{i_r}\big)
A^{1/2}.
$
Therefore,
\begin{align*}
\big(&\sum_{r=1}^kP_{i_r}\big)\big[\big(\sum_{r=1}^kP_{i_r}\big)^2+\sum_{t=1}^lP_{j_t}\big]^{-1}
\big(\sum_{r=1}^kP_{i_r}\big)\\
&=A^{1/2}\big(\sum\limits_{r=1}^kQ_{i_r}\big)\big(\big[\big(\sum\limits_{r=1}^kQ_{i_r}\big)A
\big(\sum\limits_{r=1}^kQ_{i_r}\big)\big]^{-1}+\sum\limits_{t=1}^lQ_{j_t}\big)\big(\sum\limits_{r=1}^kQ_{i_r}\big)A^{1/2}\\
&=P_{i_1}\vee\cdots \vee P_{i_k}
\end{align*}
by (\ref{3eqb}).
\end{proof}

\section{Perturbations of a complete $n$--tuple of projections}
\setcounter{equation}{0}

Recall that for any non--zero operator $C\in B(H)$, the reduced minimum modulus $\gamma(C)$ is given by
$\gamma(C)=\{\|Cx\|\,\vert\,x\in(\Ker(C))^\perp,\,\|x\|=1\}$ (cf. \cite[Remark 1.2.10]{Xue}).
We list some properties of the reduced minimum modulus as our lemma as follows.

\begin{lemma}[\rm cf. \cite{Xue}]\label{lem4b}
Let $C$ be in $B(H)\backslash\{0\}$, Then
\begin{enumerate}
\item[$(1)$] $\|Cx\|\ge\gamma(C)\|x\|$, $\forall\,x\in(\Ker(C))^\perp$.
\item[$(2)$] $\gamma(C)=\inf\{\lambda\,\vert\,\lambda\in\sigma(|C|)\backslash\{0\}\}$, where $|C|=(C^*C)^{1/2}$.
\item[$(3)$] $\gamma(C)>0$ iff $\Ran(C)$ is closed iff $0$ is an isolated point of $\sigma(|C|)$ if $0\in\sigma(|C|)$.
\item[$(4)$] $\gamma(C)=\|C^{-1}\|^{-1}$ when $C$ is invertible.
\item[$(5)$] $\gamma(C)\ge\|B\|^{-1}$ when $CBC=C$ for $B\in B(H)\backslash\{0\}$.
\end{enumerate}
\end{lemma}

For $a\in\aa_+$, put $\beta(a)=\inf\{\lambda\vert\,\lambda\in\sigma(a)\backslash\{0\}\}$. Combining Lemma \ref{lem4b}
with the faithful representation of $\aa$, we can obtain
\begin{corollary}\label{coro4a0}
Let $a\in\aa_+$. Then
\begin{enumerate}
\item[$(1)$] $\beta(a)>0$ if and only if $0\in\sigma(a)$ is isolated when $a\not\in GL(\aa)$.
\item[$(2)$] $\beta(c)\ge \|c\|^{-1}$ when $aca=a$ for some $c\in\aa_+\backslash\{0\}$.
\end{enumerate}
\end{corollary}

Let $\mathcal E$ be a $C^*$--subalgebra of $B(H)$ with the unit $I$. Let $(T_1,\cdots,T_n)$ be an $n$--tuple of positive
operators in $\mathcal E$ with $\Ran(T_i)$ closed, $i=1,\cdots,n$. Put
$H_0=\bigoplus\limits^n_{i=1}\Ran(T_i)\subset\bigoplus\limits^n_{i=1}H\triangleq\hat H$ and
$H_1=\bigoplus\limits^n_{i=1}\Ker(T_i)\subset\hat H$. Since $H=\Ran(T_i)\oplus\Ker(T_i)$, $i=1,\cdots,n$,
it follows that $H_0\oplus H_1=\hat H$. Put $T_{ij}=T_iT_j\big\vert_{\Ran(T_j)}$, $i,j=1,\cdots,n$ and set
\begin{equation}\label{4eqa}
T=\begin{bmatrix}T_1^2&T_1T_2&\cdots&T_1T_n\\ T_2T_1&T_2^2&\cdots&T_2T_n\\ \cdots&\cdots&\cdots&\cdots\\
T_nT_1&T_2T_2&\cdots&T_n^2\end{bmatrix}\in\m_n(\mathcal E),\
\hat T=\begin{bmatrix}T_{11}&T_{12}&\cdots&T_{1n}\\ T_{21}&T_{22}& \cdots & T_{2n}\\ \cdots&\cdots&\cdots&\cdots\\
T_{n1}& T_{n2}&\cdots& T_{nn}\end{bmatrix}\in B(H_0),
\end{equation}
Clearly, $H_1\subset\Ker(T)$ and it is easy to check that $\Ker(T)=H_1$ when $\Ker(\hat T)=\{0\}$. Thus, in this case,
$T$ can be expressed as $T=\begin{bmatrix}\hat T&0\\ 0&0\end{bmatrix}$ with respect to the orthogonal decomposition
$\hat H=H_0\oplus H_1$ and consequently, $\sigma(T)=\sigma(\hat T)\cup\{0\}$.

\begin{lemma}\label{lem4a}
Let $(T_1,\cdots,T_n)$ be an $n$--tuple of positive operators in $\mathcal E$ with $\Ran(T_i)$ closed, $i=1,\cdots,n$.
Let $H_0,H_1,\hat H$ be as above and $T,\,\hat T$ be given in (\ref{4eqa}). Suppose that $\hat T$ is invertible in
$B(H_0)$. Then
\begin{enumerate}
\item[$(1)$] $\sigma(\hat T)=\sigma\big(\sum\limits^n_{i=1}T_i^2\big)\backslash\{0\}$.
\item[$(2)$] $0$ is an isolated point in $\sigma\big(\sum\limits^n_{i=1}T_i\big)$
if $0\in\sigma\big(\sum\limits^n_{i=1}T_i\big)$.
\item[$(3)$] $\{T_1a_1,\cdots,T_na_n\}$ is linearly independent for any $a_1,\cdots,a_n\in\mathcal E$ with $T_ia_i\not=0$,
$i=1,\cdots,n$.
\end{enumerate}
\end{lemma}
\begin{proof} (1) Put $Z=\begin{bmatrix}T_1&\cdots&T_n\\ 0&\cdots&0\\ \vdots&\ddots&\vdots\\ 0&\cdots&0\end{bmatrix}\in\m_n(\mathcal E)$.
Then $ZZ^*=\begin{bmatrix}\sum\limits^n_{i=1}T_i^2&\\ \ &0\\ \ &\ &\ddots\\ \ &\ &\ &0\end{bmatrix}$ and
$Z^*Z=T$. Thus, $\sigma\big(\sum\limits^n_{i=1}T_i^2\big)\backslash\{0\}=\sigma(T)\backslash\{0\}=\sigma(\hat T)$.

(2) According to (1), $0$ is an isolated point of $\sigma\big(\sum\limits^n_{i=1}T_i^2\big)$ if $\sum\limits^n_{i=1}T_i^2$
is not invertible in $\mathcal E$. So by Lemma \ref{lem3a}, there is $G\in\mathcal E_+$ such that
$$
\big(\sum\limits^n_{i=1}T_i^2\big)G\big(\sum\limits^n_{i=1}T_i^2\big)=\sum\limits^n_{i=1}T_i^2,\
G\big(\sum\limits^n_{i=1}T_i^2\big)G=G,\ \big(\sum\limits^n_{i=1}T_i^2\big)G=G\big(\sum\limits^n_{i=1}T_i^2\big).
$$
Put $P_0=I-\big(\sum\limits^n_{i=1}T_i^2\big)G\in\mathcal E$. Then $P_0$ is a projection
with $\Ran(P_0)=\Ker\big(\sum\limits^n_{i=1}T_i^2\big)$. Noting that $\Ker\big(\sum\limits^n_{i=1}T_i^2\big)=
\Ker\big(\sum\limits^n_{i=1}T_i\big)=\bigcap\limits^n_{i=1}\Ker(T_i)$, $\sum\limits^n_{i=1}T_i^2\in
GL((I-P_0)\mathcal E(I-P_0))$ with the inverse $G$ and $\sum\limits^n_{i=1}T^2_i\le(\max\limits_{1\le i\le n}\|T_i\|)
\sum\limits^n_{i=1}T_i$, we get that $\sum\limits^n_{i=1}T_i$ is invertible in $(I-P_0)\mathcal E(I-P_0)$. Thus,
$0$ is an isolated point of $\sigma\big(\sum\limits^n_{i=1}T_i\big)$ when $0\in\sigma\big(\sum\limits^n_{i=1}T_i\big)$.

(3) By Lemma \ref{lem4b} (3) and Lemma \ref{lem3a}, there is $T_i^\dag\in\mathcal E_+$ such that $T_iT^\dag_iT_i=T_i$,
$T^\dag_iT_iT^\dag_i=T_i^\dag$, $T^\dag_iT_i=T_iT_i^\dag$, $i=1,\cdots,n$. Thus, $\Ran(T_i)=\Ran(T_iT_i^\dag)$, $i=1,
\cdots, n$.

Let $a_1,\cdots,a_n\in\mathcal E$ with $T_ia_i\not=0$, $i=1,\cdots,n$ such that $\sum\limits^n_{i=1}\lambda_i T_ia_i=0$
for some $\lambda_1,\cdots,\lambda_n\in\mathbb C$. For any $\xi\in H$, put $x=\bigoplus\limits^n_{i=1}\lambda_i
T_iT_i^\dag a_i\xi\in H_0$. Then $\hat Tx=0$ and $x=0$ since $\hat T$ is invertible. Thus, $\lambda_iT_iT_i^\dag a_i\xi
=0$, $\forall\,\xi\in H$ and hence $\lambda_i=0$, $i=1,\cdots,n$.
\end{proof}

The following result duo to Levy and Dedplanques is very useful in Matrix Theory:
\begin{lemma}[cf. \cite{RR}]\label{lem4c}
Suppose complex $n\times n$ self--adjoint matrix $C=[c_{ij}]_{n\times n}$ is strictly diagonally dominant, that is,
$\sum\limits_{j\not=i}|c_{ij}|<c_{ii}$, $i=1,\cdots,n$. Then $C$ is invertible and positive.
\end{lemma}

\begin{proposition}\label{prop4a}
Let $T_1,\cdots,T_n\in\aa_+$. Assume that
\begin{enumerate}
\item[$(1)$] $\gamma=\min\{\beta(T_1),\cdots,\beta(T_n)\}>0$ and
\item[$(2)$] there exists $\rho\in (0,\gamma]$ such that $\eta=\max\{\|T_iT_j\|\,\vert\,i\not=j, i,j=1,\cdots,n\}<$
$(n-1)^{-1}\rho^2$.
\end{enumerate}
Then for any $\de\in [\eta,(n-1)^{-1}\rho^2)$, we have
\begin{enumerate}
\item[$(1)$] $\sigma\big(\sum\limits^n_{i=1}T_i^2\big)\backslash\{0\}\subset[\rho^2-(n-1)\de,\rho^2+(n-1)\de]$.
\item[$(2)$] $0$ is an isolated point of $\sigma\big(\sum\limits^n_{i=1}T_i\big)$ if $0\in\sigma\big(\sum\limits^n_{i=1}T_i\big)$.
\item[$(3)$] $\big(\sum\limits^n_{i=1}T_i\big)\aa=T_1\aa\dotplus\cdots\dotplus T_n\aa$.
\end{enumerate}
\end{proposition}

\begin{proof} (1) Let $(\psi ,H_\psi )$ be a faithful representation of $\aa$ with $\psi (1)=I$. We may assume that $H=H_\psi $ and
$\mathcal E=\psi (\aa)$. Put $S_i=\psi (T_i)$, $S_{ij}=S_iS_j\big\vert_{\Ran(S_j)}$, $i,j=1,\cdots,n$. Then
$\max\{\|S_iS_j\|\,\vert\,1\le i\not=j\le n\}=\eta$ and $\gamma(S_i)=\beta(T_i)$ by Lemma \ref{lem4b}, $1\le i\le n$.
Set $H_0=\bigoplus\limits^n_{i=1}\Ran(S_i)$ and
$$
\hat S=\begin{bmatrix}S_{11}&S_{12}&\cdots&S_{1n}\\ S_{21}&S_{22}& \cdots & S_{2n}\\ \cdots&\cdots&\cdots&\cdots\\
S_{n1}& S_{n2}&\cdots& S_{nn}\end{bmatrix}\in B(H_0),\quad
S_0=\begin{bmatrix}\rho^2-\lambda&-\|S_{12}\|&\cdots&-\|S_{1n}\|\\ -\|S_{21}\|&\rho^2-\lambda& \cdots &-\|S_{2n}\|\\
 \cdots&\cdots&\cdots&\cdots\\ -\|S_{n1}\|& -\|S_{n2}\|&\cdots&\rho^2-\lambda\end{bmatrix},
$$
Then for any $\lambda<\rho^2-(n-1)\de$, we have
$\sum\limits_{j\not=i}\|S_{ij}\|\le(n-1)\eta<\rho^2-\lambda.$ It follows from Lemma \ref{lem4c} that $S_0$ is positive
and invertible. Therefore the quadratic form
$$
f(x_1,x_2,\cdots,x_n)=\sum_{i=1}^n(\rho^2-\lambda)x_i^2-2\sum_{1\le i<j\le n}\|S_{ij}\|x_ix_j
$$
is positive definite and hence there exists $\alpha>0$ such that for any $(x_1,\cdots,x_n)\in{\mathbb R^n}$,
$$
f(x_1,\cdots,x_n)\ge \alpha(x_1^2+\cdots+x_n^2).
$$
So for any $\xi=\bigoplus\limits^n_{i=1}\xi_i\in H_0$, $\|S_i\xi_i\|\ge\gamma(S_i)\|\xi_i\|\ge\rho\|\xi_i\|$,
$\xi_i\in\Ran(S_i)=(\Ker(S_i))^\perp$, $i=1,\cdots,n$, by Lemma \ref{lem4b} and
\begin{align*}
<(\hat S-\lambda I)\xi,\xi>&
=\sum_{i=1}^n\|S_i\xi_i\|^2-\sum_i^n\lambda\|\xi_i\|^2+\sum_{1\le i<j\le n}(<S_{ij}\xi_j,\xi_i>+<S_{ij}^*\xi_i,\xi_j>)\\
&\ge\sum_{i=1}^n(\rho^2-\lambda)\|\xi_i\|^2-2\sum_{1\le i<j\le n}\|S_{ij}\|\|\xi_i\|\|\xi_j\|\\
&=f(\|\xi_1\|,\cdots,\|\xi_k\|)\ge\alpha\sum_{i=1}^k\|\xi_i\|^2.
\end{align*}
Therefore, $\hat S-\lambda I$ is invertible.

Similarly, for any $\lambda>\rho^2+(n-1)\de$, we can obtain that $\lambda I-\hat S$ is invertible. So
$\sigma(\hat S)\subset[\rho^2-(n-1)\de,\rho^2+(n-1)\de]\subset (0,\rho^2+(n-1)\de]$ and consequently,
$\sigma\big(\sum\limits^n_{i=1}T_i^2\big)\backslash\{0\}=\sigma\big(\sum\limits^n_{i=1}S_i^2\big)\backslash\{0\}
\subset[\rho^2-(n-1)\de,\rho^2+(n-1)\de]$ by Lemma \ref{lem4a}.

(2) Since $\sigma\big(\sum\limits^n_{i=1}T_i\big)=\sigma\big(\sum\limits^n_{i=1}S_i\big)$, the assertion follows
from Lemma \ref{lem4a} (2).

(3) By (2) and Lemma \ref{lem3a}, $\big(\sum\limits^n_{i=1}T_i\big)^\dag\in\aa$ exists.
Set $E=\big(\sum\limits^n_{i=1}T_i\big)\big(\sum\limits^n_{i=1}T_i\big)^\dag$. Obviously,
$E\aa=\big(\sum\limits^n_{i=1}T_i\big)\aa\subset T_1\aa+\cdots+T_n\aa$ for $E\big(\sum\limits^n_{i=1}T_i\big)=
\sum\limits^n_{i=1}T_i$.

From $T_i\le\sum\limits^n_{i=1}T_i$, we get that $(1-E)T_i(1-E)\le (1-E)\big(\sum\limits^n_{i=1}T_i\big)(1-E)=0$,
i.e., $T_i=ET_i$, $i=1,\cdots,n$. So $T_i\aa\subset E\aa$, $i=1,\cdots,n$ and hence
$$
T_1\aa+\cdots+T_n\aa\subset E\aa=\big(\sum\limits^n_{i=1}T_i\big)\aa\subset T_1\aa+\cdots+T_n\aa.
$$

Since for any $a_1,\cdots,a_n\in\aa$ with $T_ia_i\not=0$, $i=1,\cdots,n$, $\{S_1\psi (a_i),\cdots,S_n\psi(a_n)\}$
is linearly independent in $\mathcal E$, we have $\{T_1a_1,\cdots,T_na_n\}$ is linearly independent in $\aa$.
Therefore, $\big(\sum\limits^n_{i=1}T_i\big)\aa=E\aa=T_1\aa\dotplus\cdots\dotplus T_n\aa$.
\end{proof}

Let $P_1,P_2$ be projections on $H$. Buckholtz shows in \cite{BU} that $\Ran(P_1)\dotplus \Ran(P_2)=H$ iff
$\|P_1+P_2-I\|<1$. For $(P_1,\cdots,P_n)\in\P_n(\aa)$, we have

\begin{corollary}\label{coro4a}
Let $(P_1,\cdots,P_n)\in\P_n(\aa)$ with $\big\|\sum\limits^n_{i=1}P_i-1\big\|<(n-1)^{-2}$. Then $\pn$ is
complete in $\aa$.
\end{corollary}
\begin{proof}For any $i\not= j$,
\begin{align*}
\|P_iP_j\|^2&=\|P_iP_jP_i\|\le\big\|P_i\big(\sum_{k\not=i}P_k\big)P_i\big\|\\
&=\big\|P_i\big(\sum_{k=1}^nP_k-1\big)P_i\big\|\le\big\|\sum_{k=1}^nP_k-1\big\|<\frac{1}{(n-1)^2}.
\end{align*}
Thus $\|P_iP_j\|<(n-1)^{-1}$. Noting that
$$
\rho=\min\{\beta(P_1),\cdots,\beta(P_n)\}=1,\
\eta=\max\{\|P_iP_j\|\,\vert\,1\le i<j\le n\}<\frac{1}{n-1},
$$
we have $\big(\sum\limits^n_{i=1}P_i\big)\aa=P_1\aa\dotplus\cdots\dotplus P_n\aa$ by Proposition \ref{prop4a}.

From $\big\|\sum\limits^n_{i=1}P_i-1\big\|<(n-1)^{-2}$, we have $\sum\limits_{i=1}^nP_i$ is invertible in $\aa$ and
so that $\aa=P_1\aa\dotplus\cdots\dotplus P_n\aa$. Thus, $\pn$ is complete in $\aa$.
\end{proof}

Combing Corollary \ref{coro4a} with Theorem \ref{th1} (3), we have
\begin{corollary}
Let $P_1,\cdots,P_n$ be non--trivial projections in $B(H)$ with $\big\|\sum\limits^n_{i=1}P_i-I\big\|<(n-1)^{-2}$. Then
$H=\Ran(P_1)\dotplus\cdots\dotplus\Ran(P_n)$.
\end{corollary}

Let $(P_1,\cdots,P_n)\in\P_n(\aa)$. A well--known statement says: ``for any $\ep>0$, there is $\de>0$ such that
if $\|P_iP_j\|<\de$, $i\not=j$, $i,j=1,\cdots,n$, then there are mutually orthogonal projections $P'_1,\cdots,P'_n
\in\aa$ with $\|P_i-P_i'\|<\ep$, $i=1,\cdots,n$". It may be the first time appeared in Glimm's paper \cite{Gl}.
By using the induction on $n$, he gave its proof. But how $\de$ depends on $\ep$ is not given. Lemma 2.5.6 of
\cite{HL} states this statement and the author gives a slightly different proof. We can find from the proof of
\cite[Lemma 2.5.6]{HL} that the relation between $\de$ and $\ep$ is $\de\le\dfrac{\ep}{(12)^{(n-1)}n!}$.

The next corollary will give a new proof of this statement with the relation $\delta=\dfrac{\epsilon}{2(n-1)}$
for $\ep\in (0,1)$.

\begin{corollary}\label{coro4c}
Let $(P_1,\cdots,P_n)\in\P_n(\aa)$. Given $\ep\in (0,1)$. If $P_1,\cdots,P_n$ satisfy condition: $\|P_iP_j\|<\de
=\dfrac{\epsilon}{2(n-1)}$, $1\le i<j\le n$, then there are mutually
orthogonal projections $P'_1,\cdots,P'_n\in\aa$ such that $\|P_i-P_i'\|<\ep$, $i=1,\cdots,n$.
\end{corollary}
\begin{proof} Set $A=\sum\limits^n_{i=1}P_i$. Noting that $\gamma=\min\{\beta(P_1),\cdots,\beta(P_n)\}=1$,
$\|P_iP_j\|<\dfrac{1}{n-1}$, $1\le i<j\le n$ and taking $\rho=1$, we have
$\sigma(A)\backslash\{0\}\subset [1-(n-1)\delta,1+(n-1)\delta]$ by Proposition \ref{prop4a} (1). So the positive element
$A^\dag$ exists by Lemma \ref{lem3a}. Set $P=A^\dag A=AA^\dag\in\aa$. From $AA^\dag A=A$ and $A^\dag AA^\dag=A^\dag$,
we get that $P_i\le P$, $i=1,\cdots,n$ and $AP=PA=A$, $A^\dag P=PA^\dag=A^\dag$. So $A\in GL(P\aa P)$ with the
inverse $A^\dag\in P\aa P$. Let $\sigma_{P\aa P}(A^\dag)$ stand for the spectrum of $A^\dag$ in $P\aa P$. Then
\begin{align}
\notag\sigma_{P\aa P}(A^\dag)&=\sigma(A^\dag)\backslash\{0\}=\{\lambda^{-1}\vert\,\lambda\in\sigma(A)\backslash\{0\}\}\\
\label{4eqb}&\subset [(1+(n-1)\delta)^{-1},(1-(n-1)\delta)^{-1}],
\end{align}

Now by Proposition \ref{prop4a}, $P\aa=A\aa=P_1\aa\dotplus\cdots\dotplus P_n\aa$. Thus, by using $P_i\le P$, $i=1,\cdots,n$,
we have $P\aa P=P_1(P\aa P)\dotplus\cdots\dotplus P_n(P\aa P)$ and then $P_iA^\dag P_j=\delta_{ij}P_i$, $i,j=1,\cdots,n$
by Theorem \ref{th1}. Put $P_i'=(A^\dag)^{1/2}P_i(A^\dag)^{1/2}\in\aa$, $i=1,\cdots,n$. Then $P_1',\cdots,P_n'$ are mutually
orthogonal projections and moreover, for $1\le i\le n$,
\begin{align}
\notag\|P'_i-P_i\|&\le\|(A^\dag)^{1/2}P_i(A^\dag)^{1/2}-P_i(A^\dag)^{1/2}\|+\|P_i(A^\dag)^{1/2}-P_i\|\\
\label{4eqc}&\le(\|(A^\dag)^{1/2}\|+1)\|(A^\dag)^{1/2}-P\|.
\end{align}

Note that $0<(n-1)\de< 1/2$. Applying Spectrum Mapping Theorem to (\ref{4eqb}), we get that
$\|(A^\dag)^{1/2}\|)\le(1-(n-1)\delta)^{-1/2}<\sqrt{2}$ and
$$
\|P-(A^\dag)^{1/2}\|\le (1-(n-1)\delta)^{-1/2}-1<\frac{2}{1+\sqrt{2}}\,(n-1)\de.
$$
Thus $\|P'_i-P_i\|<2(n-1)\de=\ep$ by (\ref{4eqc}).
\end{proof}

Applying Theorem \ref{th1} and Corollary \ref{coro4c} to an $n$--tuple of linear independent unit vectors,
we have:
\begin{corollary}\label{coro4d}
Let $(\alpha_1,\cdots,\alpha_n)$ be an $n$--tuple of linear independent unit vectors in Hilbert space $H$.
\begin{enumerate}
\item[$(1)$] There is an invertible, positive operator $K$ in $B(H)$ and an $n$--tuple of mutually orthogonal unit vectors
$(\gamma_1,\cdots,\gamma_n)$ in $H$ such that $\gamma_i=K\alpha_i$, $i=1,\cdots,n$.
\item[$(2)$] Given $\ep\in (0,1)$. If $|<\alpha_i,\alpha_j>|<\dfrac{\epsilon}{2(n-1)}$, $1\le i<j\le n$,
then there exists an $n$--tuple of mutually orthogonal unit vectors
$(\beta_1,\cdots,\beta_n)$ in $H$ such that $\|\alpha_i-\beta_j\|<2\ep$, $i=1,\cdots,n.$
\end{enumerate}
\end{corollary}
\begin{proof} Set $H_1=\mathrm{span}\{\alpha_1,\cdots,\alpha_n\}$ and $P_i\xi=<\xi,\alpha_i>\alpha_i$,
$\forall\,\xi\in H_1$, $i=1,\cdots,n$. Then $(P_1,\cdots,P_n)\in\P_n(B(H_1))$ and
$\Ran(P_1)\dotplus\cdots\dotplus\Ran(P_n)=H_1$.

By Theorem \ref{th1}, $A_0=\sum\limits^n_{i=1}P_i$ is invertible in $B(H_1)$ and $P_iA_0^{-1}P_j=\delta_{ij}P_i$,
$i,j=1,\cdots,n$. Put $K=A_0^{-1/2}+P_0$ and $\gamma_i=A_0^{-1/2}\alpha_i$, $i=1,\cdots,n$, where $P_0$ is the
projection of $H$ onto $H_1^\perp$. It is easy to check that $K$ is invertible and positive in $B(H)$ with
$\gamma_i=K\alpha_i$, $i=1,\cdots,n$ and $(\gamma_1,\cdots,\gamma_n)$ is an $n$--tuple of mutually orthogonal unit
vectors. This proves (1).

(2) Note that $\|P_iP_j\|=|<\alpha_i,\alpha_j>|<\dfrac{\ep}{2(n-1)}$, $1\le i<j\le n$. Thus, by Corollary \ref{coro4c},
there are mutually orthogonal projections $P'_1,\cdots,P'_n\in\aa$ such that $\|P_i-P_i'\|<\ep$, $i=1,\cdots,n$.
Put $\beta_i'=P'_i\alpha_i$, $i=1,\cdots,n$. Then $\beta_1',\cdots,\beta_n'$ are mutually orthogonal and
$\|\alpha_i-\beta'_i\|<\epsilon$, $i=1,\cdots,n$. Set $\beta_i=\|\beta_i'\|^{-1}\beta'_i$, $i=1,\cdots,n$. Then
$<\beta_i,\beta_j>=\delta_{ij}\beta_i$, $i,j=1,\cdots,n$ and
$$
\|\alpha_i-\beta_i\|\le\|\alpha_i-\beta'_i\|+|1-\|\beta'_i\||<2\ep,
$$
for $i=1,\cdots,n$.
\end{proof}

Now we give a simple characterization of the completeness of a given $n$--tuple of projections in $C^*$--algebra $\aa$
as follows.
\begin{theorem}\label{th2}
Let $P_1,\cdots,P_n$ be projections in $\aa$. Then $(P_1,\cdots,P_n)$ is complete if and only if
$A=\sum\limits_{i=1}^nP_i$ is invertible in $\aa$ and $\|P_iA^{-1}P_j\|<\big[(n-1)\|A^{-1}\|\|A\|^2\big]^{-1}$,
$\forall\,i\not=j$, $i,j=1,\cdots,n$.
\end{theorem}
\begin{proof}
If $\pn$ is complete, then by Theorem \ref{th1}, $A$ is invertible in $\aa$ and
$P_iA^{-1}P_j=0$, $\forall\,i\not=j$, $i,j=1,\cdots,n$.

Now we prove the converse.

Put $T_i=A^{-1/2}P_iA^{-1/2}$, $1\le i\le n$. Then $\sum\limits^n_{i=1}T_i=1$. Since $T_i(A^{1/2}P_iA^{1/2})T_i=T_i$,
we have $\beta(T_i)\ge\|A^{1/2}P_iA^{1/2}\|^{-1}\ge\|A\|^{-1}$, $i=1,\cdots,n$ by Corollary \ref{coro4a0}. Put
$\rho=\|A\|^{-1}$. Then
$$
\|T_iT_j\|\le\|A^{-1}\|\|P_iA^{-1}P_j\|<\big[(n-1)\|A\|^2\big]^{-1}=\frac{\rho^2}{n-1},\ i\not=j,\ i,j=1,\cdots,n.
$$
Thus by Proposition \ref{prop4a} (3), $\aa=T_1\aa\dotplus\cdots\dotplus T_n\aa$. Note that $T_i\aa=A^{-1/2}(P_i\aa)$,
$i=1,\cdots,n$. So $P_1\aa\dotplus\cdots\dotplus P_n\aa=A^{1/2}\aa=\aa$,
i.e., $\pn\in\PC_n(\aa)$.
\end{proof}

\begin{corollary}\label{coro4b}
Let $(P_1,\cdots,P_n)\in\PC_n(\aa)$ and let  $(P'_1,\cdots,P'_n)\in\P_n(\aa)$. Assume that
$\|P_i-P'_i\|<\big[4n^2(n-1)\|A^{-1}\|^2(n\|A^{-1}\|+1)\big]^{-1}$, $i=1,\cdots,n$, where $A=\sum\limits^n_{i=1}P_i$, then
$(P'_1,\cdots,P'_n)\in\PC_n(\aa)$.
\end{corollary}
\begin{proof} Set $B=\sum\limits^n_{i=1}P'_i$. Since $n\|A^{-1}\|\ge\|A\|\|A^{-1}\|\ge 1$, it follows that
$\|A-B\|<\dfrac{1}{2\|A^{-1}\|}$. Thus $B$ is invertible in $\aa$ with
$$
\|B^{-1}\|\le\dfrac{\|A^{-1}\|}{1-\|A^{-1}\|\|A-B\|}<2\|A^{-1}\|,\
\|B^{-1}-A^{-1}\|<2\|A^{-1}\|^2\|A-B\|.
$$
Note that $P_iA^{-1}P_j=0$, $i\not=j$, $i,j=1,\cdots,n$, we have
\begin{align*}
\|P'_iB^{-1}P'_j\|&\le\|P'_i(B^{-1}-A^{-1})P'_j\|+\|(P'_i-P_i)A^{-1}P'_j\|+\|P_iA^{-1}(P_j-P'_j)\|\\
&\le 2\|A^{-1}\|^2\|A-B\|+\|A^{-1}\|\|P_i-P'_i\|+\|A^{-1}\|\|P_j-P'_j\|\\
&<\frac{1}{2n^2(n-1)\|A^{-1}\|}<\frac{1}{(n-1)\|B^{-1}\|\|B\|^2}.
\end{align*}
So $(P_1',\cdots,P_n')$ is complete in $\aa$ by Theorem \ref{th2}.
\end{proof}

\section{Some equivalent relations and topological properties on $\PC_n(\aa)$}
\setcounter{equation}{0}

Let $\aa$ be a $C^*$--algebra with the unit $1$ and let $GL_0(\aa)$ (resp. $U_0(\aa)$) be the connected component of $1$
in $GL(\aa)$ (resp. in $U(\aa)$). Set
\begin{align*}
\PI_n(\aa)&=\big\{(P_1,\cdots,P_n)\in\P_n(\aa)\,\vert\,\sum\limits^n_{i=1}P_i\in GL(\aa)\big\}\\
\PO_n(\aa)&=\big\{(P_1,\cdots,P_n)\in\P_n(\aa)\,\vert\,\sum\limits^n_{i=1}P_i=1,\ P_iP_j=0,\ i\not=j,\ i,j=1,\cdots,n\big\}.
\end{align*}

\begin{definition}\label{def5a}
Let $(P_1,\cdots,P_n)$ and $(P_1',\cdots,P_n')$ be in $\PC_n(\aa)$.
\begin{enumerate}
\item[$(1)$] We say $(P_1,\cdots,P_n)$ is equivalent to $(P_1',\cdots,P_n')$, denoted by
$(P_1,\cdots,P_n)\sim(P_1',\cdots,P_n')$, if there are $U_1,\cdots,U_n\in\aa$ such that
$P_i=U_i^*U_i$, $P_i'=U_iU_i^*$.
\item[$(2)$] $(P_1,\cdots,P_n)$ and $(P_1',\cdots,P_n')$ are called to be unitarily equivalent, denoted by
$(P_1,\cdots,P_n)\sim_u(P_1',\cdots,P_n')$, if there is $U\in U(\aa)$ such that $UP_iU^*=P_i'$, $i=1,\cdots,n$.
\item[$(3)$] $(P_1,\cdots,P_n)$ and $(P_1',\cdots,P_n')$ are called homotopically equivalent, denoted by
$(P_1,\cdots,P_n)\sim_h(P_1',\cdots,P_n')$, if there exists a continuous mapping
$F\colon [0,1]\rightarrow\PC_n(\aa)$ such that $F(0)=\pn$ and $F(1)=(P_1',\cdots,P_n')$.
\end{enumerate}
\end{definition}

It is well--know that
$$
\pn\sim_h(P_1',\cdots,P_n')\Rightarrow \pn\sim(P_1',\cdots,P_n')
$$
and if $U(\aa)$ is path--connected,
$$
\pn\sim_u(P_1',\cdots,P_n')\Rightarrow \pn\sim_h(P_1',\cdots,P_n').
$$

\begin{lemma}\label{lem5b}
Let $(P_1,\cdots,P_n)$ be in $\PC_n(\aa)$ and $C$ be a positive and invertible element in $\aa$ with
$P_iC^2P_i=P_i$, $\inn$. Then $(CP_1C,\cdots,CP_nC)\in\PC_n(\aa)$ and $\pn\sim_h(CP_1C,\cdots,CP_nC)$ in $\PC_n(\aa)$.
\end{lemma}
\begin{proof} From $(CP_iC)^2=CP_iC^2P_iC=CP_iC$, $1\le i\le n$, we have $(CP_1C,\cdots,CP_nC)$ $\in\P_n(\aa)$.
$(P_1,\cdots,P_n)\in\PC_n(\aa)$ implies that $A=\sum\limits^n_{i=1}P_i\in GL(\aa)$ and $P_iA^{-1}P_i=P_i$, $1\le i\le n$
by Theorem \ref{th1}. So $(CP_iC)\Big(\sum\limits_{i=1}^n(CP_iC)\Big)^{-1}(CP_iC)=CP_iA^{-1}P_iC$ and hence
$(CP_1C,\cdots,CP_nC)\in\PC_n(\aa)$ by Theorem \ref{th1}.

Put $A_i(t)=C^tP_iC^t$ and $B_i(t)=C^{-t}P_iC^{-t}$, $\forall\,t\in [0,1]$, $i=1,\cdots,n$. Then $Q_i(t)\triangleq
A_i(t)B_i(t)=C^tP_iC^{-t}$ is idempotent and $A_i(t)B_i(t)A_i(t)=A_i(t)$, $\forall\,t\in [0,1]$, $i=1,\cdots,n$. Thus
$A_i(t)\aa=Q_i(t)\aa$, $\forall\,t\in [0,1]$, $i=1,\cdots,n$.

By Lemma \ref{lem3b}, $P_i(t)\triangleq Q_i(t)(Q_i(t)+(Q_i(t))^*-1)^{-1}$ is a projection in $\aa$ satisfying $Q_i(t)P_i(t)=P_i(t)$
and $P_i(t)Q_i(t)=Q_i(t)$, $\forall\,t\in [0,1]$, $i=1,\cdots,n$. Clearly, $A_i(t)\aa=Q_i(t)\aa=P_i(t)\aa$, $\forall\,
t\in [0,1]$ and $t\mapsto P_i(t)$ is a continuous mapping from $[0,1]$ into $\aa$, $i=1,\cdots,n$. Thus, from
$$
(C^tP_1C^t)\aa\dotplus\cdots\dotplus (C^tP_nC^t)\aa=\aa,\quad \forall\,t\in [0,1],
$$
we get that $F(t)=(P_1(t),\cdots,P_n(t))\in\PC_n(\aa)$, $\forall\,t\in [0,1]$. Note that
$F\colon [0,1]\rightarrow\PC_n(\aa)$ is continuous with $F(0)=\pn$. Note that $A_i(1)=CP_iC$ is a projection with
$A_i(1)Q_i(1)=CP_iCCP_iC^{-1}=Q_i(1)$ and $Q_i(1)A_i(1)=A_i(1)$, $i=1,\cdots,n$. So $P_i(1)=A_i(1)$, $i=1,\cdots,n$
and $F(1)=(CP_1C,\cdots,CP_nC)$. The assertion follows.
\end{proof}

For $(P_1,\cdots,P_n)\in\PC_n(\aa)$, $A=\sum\limits^n_{i=1}P_i\in GL(\aa)$ and $Q_i=A^{-1/2}P_iA^{-1/2}$ is a projection
with $Q_iQ_j=0$, $i\not=j$, $i,j=1,\cdots,n$ (see Theorem \ref{th1}), that is, $(Q_1,\cdots,Q_n)\in\PO_n(\aa)$. Since
$C=A^{-1/2}$ satisfies the condition given in Lemma \ref{lem5b}, we have the following:
\begin{corollary}\label{coro5a}
Let $(P_1,\cdots,P_n)\in\PC_n(\aa)$ and let $(Q_1,\cdots,Q_n)$ be as above. Then
$(P_1,\cdots,P_n)\sim_h(Q_1,\cdots,Q_n)$ in $\PC_n(\aa)$.
\end{corollary}

\begin{theorem}\label{th3}
Let $\pn$ and $(P_1',\cdots,P_n')\in\PC_n(\aa)$. Then the following statements are equivalent:
\begin{enumerate}
\item[$(1)$] $\pn\sim(P_1',\cdots,P_n')$.
\item[$(2)$] there is $D\in GL(\aa)$ such that for $1\le i\le n$, $P_iDD^*P_i=P_i$ and $P'_i=D^*P_iD$.
\item[$(3)$] there is $(S_1,\cdots,S_n)\in\PC_n(\aa)$ such that
$$
\pn\sim_u(S_1,\cdots,S_n)\sim_h(P_1',\cdots,P_n').
$$
\end{enumerate}
\end{theorem}

\begin{proof} The implication (3)$\Rightarrow$(1) is obvious. We now prove the implications (1)$\Rightarrow$(2)
and (2)$\Rightarrow$(3) as follows.

(1)$\Rightarrow$ (2) Let $U_i\in\aa$ be partial isometries such that
$U_i^*U_i=P_i$, $U_iU^*_i=P'_i$, $\inn$. Put $A=\sum\limits^n_{i=1}P_i$, $A'=\sum\limits^n_{i=1}P_i'$ and
$W=A^{-1/2}\big(\sum\limits_{i=1}^nP_iU^*_iP'_i\big)A'^{-1/2}$. Then
\begin{align*}
W^*W&=A'^{-1/2}\big(\sum_{i=1}^nP'_iU_iP_i\big)A^{-1}\big(\sum_{i=1}^nP_iU^*_iP'_i\big)A'^{-1/2}\\
&=A'^{-1/2}\big(\sum_{i=1}^nP'_iU_iP_iU^*_iP'_i\big)A'^{-1/2}
=A'^{-1/2}\big(\sum_{i=1}^nP'_i\big)A'^{-1/2}=1.
\end{align*}
Similarly, $WW^*=1$. Thus, $W\in U(\aa)$. Set $D=A^{-1/2}WA'^{1/2}\in GL(\aa)$. Then, for $1\le i\le n$,
$$
D^*P_iD=\big(\sum_{i=1}^nP'_iU_iP_i\big)A^{-1}P_iA^{-1}\big(\sum_{i=1}^nP_iU^*_iP'_i\big)
=P'_iU_iP_iU^*_iP'_i=P'_i
$$
and $P_iDD^*P_i=P_i$ follows from $(D^*P_iD)^2=D^*P_iD$.

(2)$\Rightarrow$(3) Put $U=(DD^*)^{-1/2}D$. Then $U\in U(\aa)$. Set $C=U^*(DD^*)^{1/2}U$ and $S_i=U^*P_iU$,
$1\le i\le n$. Then $(S_1,\cdots,S_n)\in\PC_n(\aa)$ with $(S_1,\cdots,S_n)$$\sim_u(P_1,\cdots,P_n)$
and $(P'_1,\cdots,P'_n)=(CS_1C,\cdots,CS_nC)$.

Since $S_iC^2S_i=U^*P_iDD^*P_iU=S_i$, $i=1,\cdots,n$, it follows from Lemma \ref{lem5b} that $(P'_1,\cdots,P'_n)
\sim_h(S_1,\cdots,S_n)$ in $\PC_n(\aa)$.
\end{proof}

\begin{proposition}\label{prop5a}
For $\P_n(\aa)$, $\PC_n(\aa)$, $\PI_n(\aa)$ and $\PO_n(\aa)$, we have
\begin{enumerate}
\item[$(1)$] $\PI_n(\aa)$ is open in $\P_n(\aa)$.
\item[$(2)$] $\PC_n(\aa)$ is a clopen subset of $\PI_n(\aa)$.
\item[$(3)$] $\PO_n(\aa)$ is a strong deformation retract of $\PC_n(\aa)$.
\item[$(4)$] $\PC_n(\aa)$ is locally connected. Thus every connected component of $\PC_n(\aa)$ is path--connected.
\item[$(5)$] $(P_1,\cdots,P_n),\, (P_1',\cdots,P_n')\in\PC_n(\aa)$ are in the same connected component iff there is
$D\in GL_0(\aa)$ such that $P_i=D^*P_i'D$, $i=1,\cdots,n$.
\end{enumerate}
\end{proposition}
\begin{proof} (1) Since $h(P_1,\cdots,P_n)=\sum\limits^n_{i=1}P_i$ is a continuous mapping from
$\P_n(\aa)$ into $\aa$ and $GL(\aa)$ is open in $\aa$, it follows that $\PI_n(\aa)=h^{-1}(GL(\aa))$ is open
in $\P_n(\aa)$.

  (2) Define $F\colon\PI_n(\aa)\rightarrow\mathbb R$ by
$$
F(P_1,\cdots,P_n)=\sum_{1\le i<j\le n}(n-1)\big\|\sum^n_{i=1}P_i\big\|^2
\big\|\big(\sum^n_{i=1}P_i\big)^{-1}\big\|\big\|P_i\big(\sum^n_{i=1}P_i\big)^{-1}P_j\|.
$$
Clearly, $F$  is continuous on $\PI_n(\aa)$. By means of Theorem \ref{th2}, we get that
$\PC_n(\aa)=F^{-1}((-1,1))$ is open in $\PI_n(\aa)$ and $\PC_n(\aa)=F^{-1}(\{0\})$ is closed in $\PI_n(\aa)$.

(3) Define the continuous mapping $r\colon\PC_n(\aa)\rightarrow\PO_n(\aa)$ by
$$
r\pn=\big(A^{-1/2}P_1A^{-1/2},\cdots,A^{-1/2}P_nA^{-1/2}\big),\quad A=\sum^n_{i=1}P_i.
$$
by Theorem \ref{th1}. Clearly, $r\pn=\pn$ when $\pn\in\PO_n(\aa)$. This means that $\PO_n(\aa)$ is a retract of
$\PC_n(\aa)$.

For any $t\in [0,1]$ and $i=1,\cdots,n$, put
$$
H_i(P_1,\cdots,P_n,t)=A^{-t/2}P_iA^{t/2}(A^{-t/2}P_iA^{t/2}+A^{t/2}P_iA^{-t/2}-1)^{-1}.
$$
Similar to the proof of Lemma \ref{lem5b}, we have
$$
H(P_1,\cdots,P_n,t)=(H_1(P_1,\cdots,P_n,t),\cdots,H_n(P_1,\cdots,P_n,t))
$$
is a continuous mapping from $\PC_n(\aa)\times [0,1]$ to $\PC_n(\aa)$ with $H(P_1,\cdots,P_n,0)=(P_1,\cdots,P_n)$
and $H(P_1,\cdots,P_n,1)=r(P_1,\cdots,P_n)$. Furthermore, when $(P_1,\cdots,P_n)$ $\in\PO_n(\aa)$, $A=1$. In this case,
$H(P_1,\cdots,P_n,t)=(P_1,\cdots,P_n)$, $\forall\,t\in [0,1]$. Therefore, $\PO_n(\aa)$ is a strong deformation retract
of $\PC_n(\aa)$.

(4) Let $\pn\in\PC_n(\aa)$. Then by Corollary \ref{coro4b}, there is $\de\in (0,1/2)$ such that for any $(R_1,\cdots,R_n)\in
\P_n(\aa)$ with $\|P_i-R_i\|<\delta$, $1\le i\le n$, we have $(R_1,\cdots,R_n)\in\PC_n(\aa)$.

Let $(R_1,\cdots,R_n)\in\PC_n(\aa)$ with $\|P_j-R_j\|<\delta$, $i=1,\cdots,n$. put $P_i(t)=P_i$, $R_i(t)=R_i$ and
$a_i(t)=(1-t)P_i+tR_i$, $\forall\,t\in [0,1]$, $i=1,\cdots,n$. Then $P_i,R_i,a_i$ are self--adjoint elements in
$C([0,1],\aa)=\mathcal B$ and $\|P_i-a_i\|=\max\limits_{t\in [0,1]}\|P_i-a_i(t)\|<\delta$, $i=1,\cdots,n$.
It follows from \cite[Lemm 6.5.9 (1)]{Xue} that there exists a projection $f_i\in C^*(a_i)$ (the $C^*$--subalgebra of
$\mathcal B$ generated by $a_i$) such that $\|P_i-f_i\|\le\|P_i-a_i\|<\delta$, $i=1,\cdots,n$. Thus, $\|P_i-f_i(t)\|
<\delta$, $i=1,\cdots,n$ and consequently, $F(t)=(f_1(t),\cdots,f_n(t))$ is a continuous mapping of $[0,1]$ into
$\PC_n(\aa)$. Since $a_i(0)=P_i$, $a_i(1)=R_i$ and $f_i(t)\in C^*(a_i(t))$, $\forall\,t\in [0,1]$, we have
$f(0)=(P_1,\cdots,P_n)$ and $f(1)=(R_1,\cdots,R_n)$. This means that $\PC_n(\aa)$ is locally path--connected.

(5) There is a continuous path $P(t)=(P_1(t),\cdots,P_n(t))\in\PC_n(\aa)$, $\forall\,t\in [0,1]$ such that $P(0)=
(P_1,\cdots,P_n)$ and $P(1)=(P_1',\cdots,P_n')$. By \cite[Corollary 5.2.9.]{O}, there is a continuous mapping
$t\mapsto U_i(t)$ of $[0,1]$ into $U(\aa)$ with $U_i(0)=1$ such that $P_i(t)=U_i(t)P_1U_i^*(t)$, $\forall\,t\in [0,1]$
and $i=1,\cdots,n$. Set
$$
W(t)=\Big(\sum\limits^n_{i=1}P_i\Big)^{-1/2}\Big(\sum\limits^n_{i=1}P_iU_i^*(t)P_i(t)\Big)
\Big(\sum\limits^n_{i=1}U_i(t)P_iU_i^*(t)\Big)^{-1/2}
$$
and $D(t)=\Big(\sum\limits^n_{i=1}P_i\Big)^{-1/2}W(t)\Big(\sum\limits^n_{i=1}U_i(t)P_iU_i^*(t)\Big)^{1/2}$, $\forall\,
t\in [0,1]$. Then $W(t)\in U(\aa)$ with $W(0)=1$, $D(t)\in GL(\aa)$ with $D(0)=1$ and $W(t)$, $D(t)$ are all continuous
on $[0,1]$ with $D^*(t)P_iD(t)=P_i(t)$ (see the proof of (1)$\Rightarrow$(2) in Theorem \ref{th3}), $\forall\,t\in [0,1]$
and $i=1,\cdots,n$. Put $D=D(1)$. Then $D\in GL_0(\aa)$ and $D^*P_iD=P_i'$, $i=1,\cdots,n$.

Conversely, if there is $D\in GL_0(\aa)$ such that $D^*P_iD=P_i'$, $i=1,\cdots,n$. Then $U=(DD^*)^{-1/2}D\in U_0(\aa)$
and $P_iDD^*P_i=P_i$, $UP_i'U^*=(DD^*)^{1/2}P_i(DD^*)^{1/2}$, $i=1,\cdots,n$. Thus,
$(P_1',\cdots,P_n')\sim_h(UP_1'U^*,\cdots,UP_n'U^*)$ and
$$
((DD^*)^{1/2}P_1(DD^*)^{1/2},\cdots,(DD^*)^{1/2}P_n(DD^*)^{1/2})\sim_h(P_1,\cdots,P_n)
$$
by Lemma \ref{lem5b}. Consequently, $(P_1',\cdots,P_n')\sim_h(P_1,\cdots,P_n)$.
\end{proof}

As ending of this section, we consider following examples:

\begin{example}
{\rm Let $\aa=\m_k(\mathbb C)$, $k\ge 2$. Define a mapping $\rho\colon\PC_n(\aa)\rightarrow\mathbb N^{n-1}$ by
$\rho(P_1,\cdots,P_n)=(\Tr(P_1),\cdots,\Tr(P_{n-1}))$, where $2\le n\le k$ and $\Tr(\cdot)$ is the canonical trace on
$\aa$.

By Theorem \ref{th1}, $(P_1,\cdots,P_n)\in\PC_n(\aa)$ means that $A=\sum\limits^n_{i=1}P_i\in GL(\aa)$ and
$(A^{-1/2}P_1A^{-1/2},\cdots,A^{-1/2}P_nA^{-1/2})\in\PO_n(\aa)$. Put $Q_i=A^{-1/2}P_iA^{-1/2}$, $i=1,\cdots,n$. Since
$(P_1,\cdots,P_n)\sim_h(Q_1,\cdots,Q_n)$ by Corollary \ref{coro5a}, it follows that $\Tr(P_i)=\Tr(Q_i)$, $i=1,\cdots,n$
and $\Tr(A)=k$. Thus $\Tr(P_n)=k-\sum\limits^{n-1}_{i=1}P_i$.

Note that $U(\aa)$ is path--connected. So, for
$(P_1,\cdots,P_n),\,(P_1',\cdots,P_n')\in\PC_n(\aa)$, $(P_1,\cdots,P_n)$ and $(P_1',\cdots,P_n')$ are in the same
connected component if and only if $\rho(P_1,\cdots,P_n)=\rho(P_1',\cdots,P_n')$.

The above shows that $\PC_k(\aa)$ is connected and $\PC_n(\aa)$ is not connected when $k\ge 3$ and $2\le n\le k-1$.
}
\end{example}

\begin{example}\label{exa4a}
{\rm Let $H$ be a separable complex Hilbert space and $\mathcal K(H)$ be the $C^*$--algebra of all compact operators
in $B(H)$. Let $\aa=B(H)/\mathcal K(H)$ be the Calkin algebra and $\pi\colon B(H)\rightarrow\aa$ be the quotient mapping.
Then $\PC_n(\aa)$ is path--connected.

In fact, if $(P_1,\cdots,P_n), (P_1',\cdots,P_n')\in\PC_n(\aa)$, then we can find $(Q_1,\cdots,Q_n)$, $(Q_1',\cdots,Q_n')
\in\PO_n(\aa)$ such that $(P_1,\cdots,P_n)\sim_h(Q_1,\cdots,Q_n)$ and $(P'_1,\cdots,P'_n)\sim_h(Q'_1,\cdots,Q'_n)$
by Corollary \ref{coro5a}. Since $B(H)$ is of real rank zero, it follows from  \cite[Corollary B.2.2]{Xue} or
\cite[Lemma 3.2]{X} that there are projections $R_1,\cdots,R_n$ and $R_1',\cdots,R_n'$ in $B(H)$ such that $\pi(R_i)=Q_i$,
$\pi(R_i')=Q_i'$, $i=1,\cdots,n$ and
$$
R_iR_j=\delta_{ij}R_i,\ R_i'R_j'=\delta_{ij}R_i',\ i,j=1,\cdots,n,\ \sum\limits^n_{i=1}R_i=\sum\limits^n_{i=1}R_i'=I.
$$
Note that $R_1,\cdots,R_n,R_1',\cdots,R_n'\not\in\mathcal K(H)$. So there are partial isometry operators
$V_1,\cdots,V_n$ in $B(H)$ such that $V_i^*V_i=R_i$, $V_iV_i^*=R_i'$, $i=1,\cdots,n$. Put $V=\sum\limits^n_{i=1}V_i$.
Then $V\in U(B(H))$ and $VR_iV^*=R_i'$, $i=1,\cdots,n$. Put $U=\pi(V)\in U(\aa)$. Then
$(UQ_1U^*,\cdots,UQ_nU^*)=(Q_1',\cdots,Q_n)$ in $\PO_n(\aa)$. Since $U(B(H))$ is
path--connected, we have $(Q_1,\cdots,Q_n)\sim_h(Q_1',\cdots,Q_n')$ in $\PC_n(\aa)$. Finally,
$(P_1,\cdots,P_n)\sim_h(P_1',\cdots,P_n')$. This means that $\PC_n(\aa)$ is path--connected.
}
\end{example}

\end{document}